\documentclass[12pt,reqno]{amsart}
\usepackage[latin1]{}
\usepackage{graphicx}
\usepackage{tikz-cd}
\usepackage{graphicx}
\usepackage{caption}
\usepackage{tabularx}
\usepackage{longtable}
\usepackage{tabularray}
\usepackage{amsmath, amssymb}
\usepackage{array}
\usepackage[english]{babel}  
\usepackage{amsmath}
\usepackage{longtable}
\usepackage{mathtools}
\usepackage{multicol}
\usetikzlibrary{positioning}
\usepackage{rotating}
\usepackage{verbatim} 

\usepackage{graphicx}
\usepackage{tikz-cd}
\usepackage{caption}

\usepackage{amssymb,tikz}
\usetikzlibrary{positioning}
\usepackage[colorlinks=false,urlbordercolor=white]{hyperref}

\usetikzlibrary{positioning,shapes.multipart,arrows.meta}

\newcommand{\comentario}[1]{}
\makeatletter
\newtheorem*{rep@theorem}{\rep@title}
\newcommand{\newreptheorem}[2]{%
\newenvironment{rep#1}[1]{%
 \def\rep@title{#2 \ref{##1}}%
 \begin{rep@theorem}}%
 {\end{rep@theorem}}}
\makeatother

\def\gn#1#2{{$\href{http://groupnames.org/\#?#1}{#2}$}}
\def\gn#1#2{$#2$}  
\tikzset{sgplattice/.style={inner sep=1pt,norm/.style={red!50!blue},char/.style={blue!50!black},
  lin/.style={black!50}},cnj/.style={black!50,yshift=-2.5pt,left=-1pt of #1,scale=0.5,fill=white}}

\newreptheorem{theorem}{Theorem}
\newreptheorem{lemma}{Lemma}

\newtheorem{theorem}{Theorem}[section]

\newtheorem{proposition}[theorem]{Proposition}
\newtheorem{corollary}[theorem]{Corollary}

\newtheorem{lemma}[theorem]{Lemma}

\newcommand{\Q}{\mathbb Q}
\newcommand{\Qbar}{{\overline{\mathbb Q}}}

\newcommand{\F}{\mathbb F}

\newcommand{\GL}{\mathrm{GL}}
\newcommand{\PGL}{\mathrm{PGL}}

\newcommand{\Aut}{\operatorname{Aut}}

\newcommand{\Gal}{\operatorname{Gal}}

\theoremstyle{definition}

\begin{document}
    
\title{
On Galois Embedding Problems Arising from 3-Torsion of Elliptic Curves
}

\author[Gálvez]{José-A. Gálvez}
\address[Gálvez]{Universitat Politècnica de Catalunya}
\author[Lario]{Joan-C. Lario}
\address[Lario]{Universitat Politècnica de Catalunya}


\date{\today. The first author  gratefully acknowledges the Universitat Politècnica de Catalunya and Banco Santander for the financial support provided through the predoctoral FPI-UPC grant 2025\_FPI-UPC\_279750.\\
The second author is funded by the project PID2022-136944NB-IOO (MECD)
}


\begin{abstract}
We study Galois embedding problems arising from the 3-torsion of elliptic curves defined over $\mathbb{Q}$, extending Lario and Rio \cite{LR}
correspondence to all possible images of mod~$3$ Galois representations; namely, 
$\GL_2(\F_3),SD_{16},D_6,D_4$ and $C_2^2$.
In the cyclotomic case, we show that solvability of these embedding problems is equivalent to the existence of infinitely many elliptic curves whose \(3\)-division fields provide the corresponding solutions.

\end{abstract}

\flushbottom
\maketitle
\tableofcontents
\thispagestyle{empty}

\section{Introduction} 
In this article, we develop a unified approach to Galois embedding problems arising from the 3-torsion of elliptic curves defined over $\mathbb{Q}$.
This relationship was first studied in \cite{LR} for the exhaustive case of the mod 3 Galois representation. We propose a generalization that also works for non-exhaustive cases.

Let $E/\mathbb{Q}$ be an elliptic curve given by a short Weierstrass equation, and let $\psi_{3,E}$ denote its 3-division polynomial.  We have that $\psi_{3,E}$ is a four degree polynomial with discriminant $\Delta_{\psi_{3,E}} \equiv -3 \pmod{\Q^{*2}}$. The splitting field of this polynomial, denoted by $\Q_x(E[3])$, is the smallest field containing the $x$-coordinates of all the 3-torsion points of $E$. The 3-division field of $E$, denoted by $\Q(E[3])$, is the smallest extension of $\Q_x(E[3])$ that also contains the $y$-coordinates of the 3-torsion points of~$E$. When the Galois group of $\psi_{3,E}$ is isomorphic to the symmetric group $S_4$, the  mod 3 Galois representation $\varrho_{3,E} \colon \Gal(\Qbar/\Q) \to \GL_2(\F_3)$ respect to some basis of $E[3]$ is a solution of the Galois embedding problem 
\begin{center}
\begin{tikzcd}
& & & \Gal(\Qbar/\Q)
\arrow[->>,d]
\arrow[->>,ldd,dotted,"\varrho"'] & \\
& & & \Gal(\Q_x(E[3])/\Q) \arrow{d}{\varphi} & \\
1 \arrow{r} & \{\pm 1\} \arrow{r} & \GL_2(\F_3)
\arrow{r}{\pi} & \PGL_2(\F_3) \cong S_4 \arrow{r} & 1,\\
\end{tikzcd}
\end{center}
where $\pi \colon \GL_2(\F_3) \to \PGL_2(\F_3)$ is the natural projection and $\varphi \colon \Gal(\Q_x(E[3])/\Q) \to \PGL_2(\F_3)$ is a group isomorphism\footnote{Note that a different choice of $\varphi$ gives an equivalent embedding problem since $\Aut(S_4) = \mathrm{Inn}(S_4)$.}. 

The study of this Galois embedding problem for an arbitrary field $K$ arising as the splitting field of a quartic polynomial $f$ with $\Delta_f \equiv -3 \pmod{\Q^{*2}}$ leads to a geometric equivalence:
\begin{theorem}[Lario and Rio, 1995]\label{th_LR}
    Let $K$ be the splitting field of a four-degree polynomial $f$ over~$\Q$ with $\Delta_f \equiv -3 \pmod{\Q^{*2}}$. Suppose that $\Gal(K/\Q) \cong \PGL_2(\F_3)$. Then the Galois embedding problem 
    \begin{center}
\begin{tikzcd}
& & & \Gal(\Qbar/\Q)
\arrow[->>,d]
\arrow[->>,ldd,dotted,"\varrho"'] & \\
& & & \Gal(K/\Q) \arrow{d}{\varphi} & \\
1 \arrow{r} & \{\pm 1\} \arrow{r} & \GL_2(\F_3)
\arrow{r}{\pi} & \PGL_2(\F_3) \cong S_4 \arrow{r} & 1,
\end{tikzcd}
\end{center}
is solvable if and only if $K$ is the splitting field of the 3-division polynomial of an elliptic curve $E/\Q$.
\end{theorem}

The natural question addressed in this paper is whether a similar correspondence persists beyond the surjective case, that is, when $\Gal(K/\Q)$ is isomorphic to a proper subgroup of $S_4$. 
More precisely, we study Galois embedding problems associated with all possible images of mod~$3$ Galois representations of elliptic curves --namely $\GL_2(\F_3),SD_{16},D_6,D_4$ and $C_2^2$-- as they arise in the classification of \cite{Zywina}. 

In Section~\ref{sec_General_linear_group}, we construct an isomorphism $\Phi \colon S_4 \to \PGL_2(\F_3)$ which allows us to consider suitable embeddings $\Gal(K/\Q) \hookrightarrow \PGL_2(\F_3)$ for the Galois embedding problem of interest. 
In Section~\ref{ECmod5}, we study the genus-zero modular curves \(X_{\widetilde{G}_i}(3)\), for appropriate subgroups ${\widetilde{G}_i}\subseteq \GL_2(\F_3)$, 
\[
\widetilde{G}_4 = \GL_2(\F_3) ,\quad
\widetilde{G}_3 \cong SD_{16},\quad
\widetilde{G}_2 \cong D_6,\quad
\widetilde{G}_1 \cong D_4,\quad
\widetilde{G}_0 \cong C_2^2, 
\]
whose rational non-cuspidal non-CM points classify elliptic curves \(E\) over \(\mathbb{Q}\) for which the image of the mod~\(3\) Galois representation on \(\Aut(E[3])\) is contained in $\widetilde{G}_i$. In Section~\ref{sec_Galois_embedding_problems}, we show our main result in this paper that generalizes Theorem \ref{th_LR}.

\begin{reptheorem}{main_theorem}
    Let $K$ be the splitting field
    of a four-degree polynomial $f$ over $\Q$ with $\Delta_f \equiv -3 \pmod{\Q^{*2}}$. 
    Then there exists an isomorphism 
    $\varphi_f\colon \Gal(K/\Q)\longrightarrow 
    \widetilde{G}_i/\{\pm 1\}$ for some $i\in \{0,1,2,3,4\}$ such that
    the Galois embedding problem 
    \begin{center}
\begin{tikzcd}
& & & \Gal(\Qbar/\Q)
\arrow[->>,d]
\arrow[->>,ldd,dotted,"\varrho"'] & \\
& & & \Gal(K/\Q) \arrow{d}{\varphi_f} & \\
1 \arrow{r} & \{\pm 1\} \arrow{r} & 
\widetilde{G}_i
\arrow{r}{\pi} & \widetilde{G}_i/\{\pm 1\} \arrow{r} & 1,
\end{tikzcd}
\end{center}
is solvable if and only if $K$ is the splitting field of the 3-division polynomial of an elliptic curve $E/\Q$.
\end{reptheorem}

We emphasize that the choice of the isomorphism $\varphi_f$ (which depends on the isomorphism introduced in Section~\ref{sec_General_linear_group}) is both delicate and essential, as replacing it with another isomorphism gives rise to different embedding problems. Finally, Section~\ref{sec_examples} illustrates the theorem with explicit examples.
In forthcoming papers, we shall deal with similar Galois embedding problems for $p>3$.

\section{The isomorphism $\Phi$ for $\PGL_2(\F_3)$}\label{sec_General_linear_group}

In order to generalize Theorem \ref{th_LR} to proper subgroups of $S_4$ we need to fix a way to embed Galois groups arising from four-degree polynomials (with discriminant $-3$ up to squares) into $\PGL_2(\F_3)$. We do this by fixing an appropriate isomorphism $\Phi \colon S_4 \to \PGL_2(\F_3)$.

In $\GL_2(\F_3)$, consider the matrices $$A_{(1 \ 2)}= \begin{pmatrix}
    0 & 1 \\ 
    1 & 0
\end{pmatrix}, \quad A_{(2 \ 3)} = \begin{pmatrix}
    2 & 1 \\ 
    0 & 1
\end{pmatrix} \quad \text{ and } \quad A_{(3 \ 4)} = \begin{pmatrix}
    2 & 0 \\ 
    0 & 1
\end{pmatrix},$$ and define the isomorphism $\Phi \colon S_4 \to \PGL_2(\F_3)$ sending the transpositions $(1 \ 2) \mapsto \pm A_{(1 \ 2)}, \ (2 \ 3) \mapsto \pm A_{(2 \ 3)}$, and $(3 \ 4) \mapsto \pm A_{( 3 \ 4)}$. It is easy to check that, for every elliptic curve $E$ with 3-division polynomial $\psi_{3,E}$ and roots $x_1,x_2,x_3,x_4$, the isomorphism $\Phi$ is chosen in a way that the following diagram commutes:
\begin{center}
\begin{tikzcd}
& & & \operatorname{Gal}(\Qbar/\Q) 
\arrow[->>,d]
\arrow[->,ldd,"\varrho_{3,E}"']& \\
& & & \operatorname{Gal}(\Q_x(E[3])/\Q) \arrow{d}{\varphi} & \\
1\arrow{r} & \{\pm 1\} \arrow{r} & \GL_2(\F_3)
\arrow{r}{\pi} & \PGL_2(\F_3)
\arrow{r} & 1, \\
\end{tikzcd}
\end{center}
where $\varphi = \Phi \circ s_{\psi_{3,E}}$, the map 
\[
s_{\psi_{3,E}} \colon \Gal(\Q_x(E[3])/\Q) \to S_4
\]
is the associated permutation representation, given by $\sigma(x_i)=x_{s_{\psi_{3,E}}(\sigma)(i)}$ for all $\sigma \in \Gal(K/\Q)$, and
\[
\varrho_{3,E}\colon \Gal(\Qbar/\Q)\to \GL_2(\F_3)
\]
is the mod $3$ Galois representation attached to $E$, with respect to the basis $\{P_1,P_2\}=\{(x_1,y_1),(x_2,y_2)\}$ of $E[3]$\footnote{Example: Suppose $\sigma \in \Gal(\Qbar/\Q)$ such that $s_{\psi_{3,E}}(\sigma_{|_K})=(1 \ 2)(3 \ 4)$, then $\sigma(P_1) \in \{ P_2,-P_2\}$, $\sigma(P_2) \in\{ P_1,-P_1\}$ and $\sigma(P_1+P_2)\in \{P_1-P_2,P_2-P_1\}$. Thus, $\varrho_{E,3}(\sigma)$ is either $\begin{pmatrix}
    0 & 2 \\
    1 & 0
\end{pmatrix}$ or $\begin{pmatrix}
    0 & 1 \\
    2 & 0
\end{pmatrix}$. In any case, $\pi \circ \varrho_{E,3}(\sigma) = \Phi \circ s_{\psi_{3,E}} \circ \pi (\sigma)$.}.



Furthermore, as we show in Section \ref{ECmod5}, the following subgroups of $\PGL_2(\F_3)$:
\[
\begin{aligned}
G_0 &= \Phi(\langle (1\ 2) \rangle) &&\cong C_2, \\
G_1 &= \Phi(\langle (1\ 2), (3\ 4) \rangle) &&\cong C_2^2, \\
G_2 &= \Phi(\langle (1\ 2\ 3), (1\ 2) \rangle) &&\cong S_3, \\
G_3 &= \Phi(\langle (1\ 2\ 3\ 4), (2\ 4) \rangle) &&\cong D_4, \\
G_4 &= \Phi(\langle (1\ 2), (2\ 3), (3\ 4) \rangle) &&\cong S_4,
\end{aligned}
\] 
are isomorphic to all the possible Galois groups arising from the $x$-coordinates of 3-torsion points of an elliptic curve over $\Q$. If we define $\widetilde{G}_i = \pi^{-1}(G_i)$, then
\[
\begin{aligned}
\widetilde{G}_0 &\cong C_2^2, \\
\widetilde{G}_1 &\cong D_4, \\
\widetilde{G}_2 &\cong D_6, \\
\widetilde{G}_3 &\cong SD_{16}, \\
\widetilde{G}_4 &= \GL_2(\F_3)
\end{aligned}
\]
and each of these pairs of groups gives rise to an exact sequence 
$$ 1 \to \{\pm I\} \to \widetilde{G}_i \to G_i \to 1.$$

In Section \ref{sec_Galois_embedding_problems}, we discuss how this sequences relates to the family of Galois embedding problems that generalizes Theorem \ref{th_LR}.

\section{The modular curves $X_{\widetilde G}(3)$}\label{ECmod5}

Let $\widetilde{G}=\widetilde{G}_i$ be one of the five subgroups of $\GL_2(\F_3)$ as in the previous section. 
In this section, we focus on the family of elliptic curves over $\Q$ such that the image of the Galois representation attached to their $3$-torsion module is conjugate to $\widetilde{G}$.

It turns out that there is a modular curve $X_{\widetilde{G}}(3)$ over $\Q$ whose
non-cuspidal non-CM rational points classify the elliptic curves $E/\Q$ such that the image of the Galois representation~$\varrho_{E,3}$ giving the action on the torsion module $E[3]$ has image conjugate to a subgroup of $\widetilde{G}$. The genus of $X_{\widetilde{G}}(3)$ is zero,
and \cite{Zywina} provides a modular Hauptmodule $t=t_{\widetilde{G}}(\tau)$ for each case; that is, the corresponding function field is $\Q(X_{\widetilde{G}}(3))=\Q(t_{\widetilde{G}}(\tau))$.

For instance, consider the case $\widetilde{G}_3 \cong SD_{16}$.
Let
$h(\tau)=
 1/3 \,\eta(\tau/3)^3
 /\eta(3\tau)^3$, where
 $$
\eta(\tau) = q^{1/24} \prod_{n=1}^{\infty} (1 - q^n), \quad q = e^{2 \pi i \tau},
 $$
 is the Dedekind eta modular function. Then, the function field of the modular curve $X_{\widetilde{G}_3}(3)$ is $\Q(t)$ where
 $$
 t = t_{\widetilde{G}_3}(\tau) = \frac{
 3 (h(\tau)+1) (h(\tau)+3) (h(\tau)^2+3)}
   {h(\tau) \left(h(\tau)^2+3
   h(\tau)+3\right)}.
 $$
The forgetful map
$j\colon X_{\widetilde{G}_3}(3)\longrightarrow X(1)$ provides the expression for the Klein modular $j$-function:
 $$
 j = j(\tau)  = t^3.
 $$
The function field extension $\Q(t)/\Q(j)$ has degree $[\GL_2(\F_3)\,\colon \widetilde{G}_3]=48/16=3$. 
In this case, there are twelve rational values \(t \in \mathbf{P}^1(\mathbb{Q})\) giving either a cuspidal point  or an elliptic curve over \(\mathbb{Q}\) with CM (complex multiplication). 

Using the Hauptmoduln provided in \cite{Zywina} for these modular curves $X_{\widetilde{G}}(3)$, we obtain the following result.

\begin{proposition}\label{prop_Zywina} Let $E/\Q$ be an elliptic curve
without~CM. 
Define the rational functions
\[
\begin{aligned}
F_1(t) &= 27\frac{(t+1)(t+9)^3}{t^3}, 
&\quad F_2(t) &= t(t^2+3t+3), \\
G_1(t) &= t^3, 
&\quad G_2(t) &= \frac{3(t+1)(t-3)}{t}, 
&\quad G_3(t) &= \frac{t^2+3t+3}{t}.
\end{aligned}
\]
Let $K$ be the splitting field of the 3-division polynomial of $E$; i.e, the smallest field containing all 3-torsion points $x$-coordinates. Then $F_1(F_2(t))=G_1(G_2(G_3(t)))$,
\[ \Gal(\Q(E[3])/\Q) \cong \begin{cases} 
      \GL_2(\F_3) & \text{if } j(E) \not\in F_1(\Q^*) \text{ and } j(E) \not \in G_1(\Q^*), \\
      S_3 \text{ or } D_6 & \text{if } j(E) \in F_1(\Q^*) \text{ and } j(E) \not\in F_1(F_2(\Q^*)), \\
      SD_{16} & \text{if } j(E) \in G_1(\Q^*) \text{ and } j(E) \not \in G_1(G_2(\Q^*)), \\
    D_4 & \text{if } j(E) \in G_1(G_2(\Q^*)) \text{ and } j(E) \not \in G_1(G_2(G_3(\Q^*))), \\
      C_2 \text{ or } C_2^2 & \text{otherwise,}\\
      
   \end{cases}
\]
and 
\[ \Gal(K/\Q) \cong \begin{cases} 
      S_4 & \text{if } j(E) \not\in F_1(\Q^*) \text{ and } j(E) \not \in G_1(\Q^*), \\
      S_3 & \text{if } j(E) \in F_1(\Q^*) \text{ and } j(E) \not\in F_1(F_2(\Q^*)), \\
      D_4 & \text{if } j(E) \in G_1(\Q^*) \text{ and } j(E) \not \in G_1(G_2(\Q^*)), \\
    C_2^2 & \text{if } j(E) \in G_1(G_2(\Q^*)) \text{ and } j(E) \not \in G_1(G_2(G_3(\Q^*))), \\
      C_2 & \text{otherwise.}\\   
   \end{cases}
\]
The ambiguous cases for $\Gal(\Q(E[3])/\Q)$ can be resolved in favor of the larger group by selecting a suitable quadratic twist of the elliptic curve $E$.
\end{proposition}
\begin{proof}
The isomorphisms involving $\Gal(\Q(E[3])/\Q)$ follows from \cite{Zywina}. For the isomorphisms involving $\Gal(K/\Q)$ define the following families of elliptic curves:
\begin{align*}
E_{t,S_4} \colon y^2 &= x^3 -3 t(t - 1728)x-2t(t - 1728)^2 \\
E_{t,S_3} \colon y^2 &= x^3 -3(t+1)(t+9)x-2(t+1)(t^2-18t-27) \\
E_{t,D_4} \colon y^2 &= x^3 -3t(t^3-1728)x - 2(t^3-1728)^2 \\
E_{t,C_2^2} \colon y^2 &= x^3-3(t - 3) (t + 1) (t^2 - 6t - 3)x-2(t^2 + 3) (t^2 - 6t - 3)^2 \\
E_{t,C_2} \colon y^2 &= x^3-3(t+1)(t+3)(t^2+3)x-2(t^2 - 3)(t^4 + 6t^3 + 18t^2 + 18t + 9).
\end{align*}
It is easy to check that 
\begin{align*}
    j(E_{t,S_4}) &= t, \\
    j(E_{t,S_3}) &= F_1(t), \\
    j(E_{t,D_4}) &= G_1(t), \\
    j(E_{t,C_2^2}) &= G_1(G_2(t)), \\
    j(E_{t,C_2}) &= F_1(F_2(t))=G_1(G_2(G_3(t)))
\end{align*}
and the 3-division polynomials of this families of elliptic curves are
\begin{align*}
    \psi_{3,E_{t,S_4}}(x) &= 3(x^4-6 t(t - 1728)x^2 + -8t(t - 1728)^2x -3 t^2(t - 1728)^2) \\
    \psi_{3,E_{t,S_3}}(x) &= 3(x-3t-3)(x^3 + 3(t + 1)x^2 + 3(t+1)(t-15)x + (t+1)(t+9)^2) \\
    \psi_{3,E_{t,D_4}}(x) &= 3x^4 - 18t(t^3-1728)x^2 -24(t^3-1728)^2x - 9t^2(t^3-1728)^2 \\
    \psi_{3,E_{t,C_2^2}}(x) &= 3(x^2-2(t^2 - 6t - 3)x+(t+1)^2(t^2-6t-3))(x^2+2(t^2 - 6t - 3)x+(t - 3)^2  (t^2 - 6t - 3)), \\
    \psi_{3,E_{t,C_2}}(x) &= 3 (x - 3t^2 - 6t - 3) (x + t^2 + 6t + 9) (x^2 + (2t^2 - 6)x + t^4 + 6t^2 + 9),
\end{align*}
respectively. 

Let $G$ be the smallest\footnote{By ``smallest'' we mean that no proper subgroup of $G$ satisfies this condition.} group such that $j(E)=j(E_{t,G})$ for some $t \in \Q$. Then $E$ is isomorphic, up to quadratic twists, to $E_{t,G}$. The Galois group of the splitting field of the 3-division polynomial of $E$ is isomorphic to the Galois group of the splitting field of $\psi_{3,E_{t,G}}$, which is clearly isomorphic to a subgroup of $G$. This, together with the fact that\footnote{This is a standard fact in the literature on elliptic curves, see \cite{bandini2016fields}.} $[\Q(E[3]):K] \leq 2$ and $\sqrt{-3} \in K$, is enough to prove $\Gal(K/\Q) \cong G$.

The last part of the theorem follows by choosing, if necessary, a quadratic twist by $d \in \Q$ $$E_{t,G}^d \colon dy^2=x^3+Ax+B$$ such that $\sqrt{d(x_1^3+Ax_1+B)} \not\in \Q_x(E_{t,G}[3])$, where $x_1$ is a root of $\psi_{3,E_{t,G}}$.

\end{proof}


\section{Galois embedding problems and 3-torsion fieds of elliptic curves}\label{sec_Galois_embedding_problems}

In this section, we extend Theorem~\ref{th_LR} to quartic polynomials whose Galois group is not necessarily isomorphic to~$S_4$. The proof is based on the family of elliptic curves introduced in Proposition~\ref{prop_Zywina}.

The embedding problems considered here are classically described in terms of a cohomological obstruction (sometimes expressed via $2$-torsion elements in the Brauer group). However, such an obstruction depends explicitly on the arithmetic of the fields involved and is often difficult to handle in practice. The main result of this section provides instead a geometric criterion: solvability is characterized by the existence of a suitable elliptic curve whose mod~$3$ Galois representation realizes the corresponding embedding problem. With the notation of Section~\ref{sec_General_linear_group}, the theorem reads as follows:



\begin{theorem}\label{main_theorem}
Let $K$ be the splitting field of  a four-degree polynomial $f$ over $\Q$ with $\Delta_f \equiv -3 \pmod{\Q^{*2}}$ . 
\begin{itemize}
\item[$a)$] One can label the four distinct roots $\alpha_1,\alpha_2,\alpha_3$, $\alpha_4$ of $f$ in a way such that the monomorphism $\varphi_f = \Phi \circ s_f$ satisfies $\varphi_f(\Gal(K/\Q))=G_i$ for some $i \in \{0,\dots,4\}$, where $s_f \colon \Gal(K/\Q) \hookrightarrow S_4$ is the permutation representation (i.e,  $\sigma(\alpha_i)=\alpha_{s_f(\sigma)(i)}$ for every $\sigma \in \Gal(K/\Q)$).
\item[$b)$] The following statements are equivalent:
\begin{itemize}
    \item[$1)$] There exists a non-CM elliptic curve $E/\Q$ such that $K$ is the splitting field of the 3 division polynomial of $E$.
    \item[$2)$] The Galois embedding problem
\begin{center}
\begin{tikzcd}
& & & \operatorname{Gal}(\Qbar/\Q) 
\arrow[->>,d]
\arrow[->>,ldd,dotted,"\varrho"']& \\
& & & \operatorname{Gal}(K/\Q) \arrow{d}{\varphi_f} & \\
1\arrow{r} & \{\pm 1\} \arrow{r} & \widetilde{G}_i
\arrow{r}{\pi} & G_i
\arrow{r} & 1. \\
\end{tikzcd}
\end{center} is solvable.
\end{itemize}
\item[$c)$] In case that $1)$ holds, the image of the mod 3 representation $\varrho_{3,E} \colon \Gal(\Qbar/\Q) \to \GL_2(\F_3)$ with respect to some basis of $E[3]$ is a subgroup of $\widetilde{G}_i$ and  $\varrho_{3,E} \colon \Gal(\Qbar/\Q) \to \widetilde{G}_i$ is a weak solution\footnote{A weak solution is a group morphism that makes the diagram commute. A proper solution is an exhaustive weak solution.} to the Galois embedding problem described at $2)$. Furthermore, this solution is proper if we take a suitable quadratic twist of $E$. 
\end{itemize}
\end{theorem}
\begin{proof}[Proof of Theorem \ref{main_theorem} $a)$.]
A different ordering of the roots of $f$ results in a conjugate of the permutation representation $s_f$. Given that the subgroup lattice of $S_4$ is
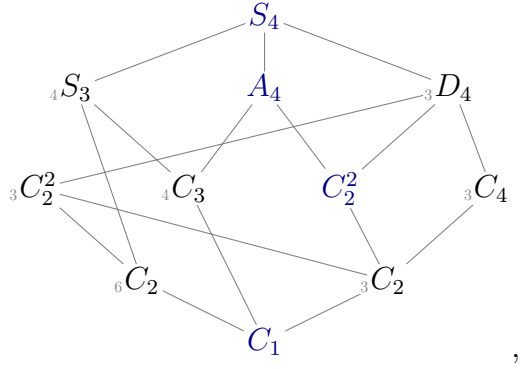
\begin{center}
   \begin{tikzpicture}[scale=1.0,sgplattice]
  \node[char] at (3.12,0) (1) {\gn{C1}{C_1}};
  \node at (4.75,0.803) (2) {\gn{C2}{C_2}};
  \node at (1.5,0.803) (3) {\gn{C2}{C_2}};
  \node at (2.12,2.02) (4) {\gn{C3}{C_3}};
  \node[char] at (4.12,2.02) (5) {\gn{C2^2}{C_2^2}};
  \node at (0.125,2.02) (6) {\gn{C2^2}{C_2^2}};
  \node at (6.12,2.02) (7) {\gn{C4}{C_4}};
  \node at (0.625,3.35) (8) {\gn{S3}{S_3}};
  \node at (5.62,3.35) (9) {\gn{D4}{D_4}};
  \node[char] at (3.12,3.35) (10) {\gn{A4}{A_4}};
  \node[char] at (3.12,4.3) (11) {\gn{S4}{S_4}};
  \draw[lin] (1)--(2) (1)--(3) (1)--(4) (2)--(5) (2)--(6) (3)--(6) (2)--(7)
     (3)--(8) (4)--(8) (5)--(9) (6)--(9) (7)--(9) (4)--(10) (5)--(10)
     (8)--(11) (9)--(11) (10)--(11);
  \node[cnj=2] {3};
  \node[cnj=3] {6};
  \node[cnj=4] {4};
  \node[cnj=6] {3};
  \node[cnj=7] {3};
  \node[cnj=8] {4};
  \node[cnj=9] {3};
\end{tikzpicture},
\end{center} any two subgroups of $S_4$ that are both isomorphic to one of $C_2,C_2^2,C_2^2,S_3,D_4$ or $S_4$ are conjugates as long as they are not subgroups of $A_4$. Hence, the first statement of Theorem \ref{main_theorem} follows just by noticing:
\begin{itemize}
    \item[1.] The image $s_f(\Gal(K/\Q))$ is not a subgroup of $A_4$ since $\Delta_f$ is not a square.
    \item[2.] The image $s_f(\Gal(K/\Q))$ is not isomorphic to $C_4$ since complex conjugation is an involution of $K$ whose image by $s_f$ is odd (but $C_4 \leq S_4$ does not contains odd involutions).
    \item[2.] Every $\Phi^{-1}(G_i)$ is not a subgroup of $A_4$.
\end{itemize}  
\end{proof}

\begin{lemma}\label{lemma_conjugates}
    Let $f$ and $g$ be two four-degree polynomials such that 
    they share the same splitting field $K$ and $\Delta_f \equiv \Delta_g \equiv -3 \pmod{\Q^{*2}}$. Then $s_f$ and $s_g$ are conjugate permutation representations.
\end{lemma}
\begin{proof}
Let $\alpha_1,\alpha_2,\alpha_3$ and $\alpha_4$ be the roots of $f$ and let $\beta_1, \beta_2, \beta_3$ and $\beta_4$ be the roots of $g$. The set of fields $S=\{\Q(\alpha_1),\Q(\alpha_2),\Q(\alpha_3),\Q(\alpha_4)\}$ is:
\begin{itemize}
    \item[1.] The set of all the four-degree subfields of $K$ if $\Gal(K/\Q) \cong S_4$.
    \item[2.] The set of all the three-degree subfields of $K$ and the field $\Q$ if $\Gal(K/\Q) \cong S_3$.
    \item[3.] The set of the two four-degree subfields of $K$ not containing $\sqrt{-3}$ if $\Gal(K/\Q) \cong D_4$.
    \item[4.] The set $\{K\}$ if $\Gal(K/\Q) \cong C_4$.
    \item[5.] The set of the two two-degree subfields of $K$ not containing $\sqrt{-3}$ if $\Gal(K/\Q) \cong C_2^2$.
    \item[6.] The set $\{\Q(\sqrt{-3}),\Q\}$ if $\Gal(K/\Q) \cong C_2$.
\end{itemize}
Thus, $S=\{\Q(\beta_1),\Q(\beta_2),\Q(\beta_3),\Q(\beta_4)\}$ as well. Therefore, there exists $\sigma \in S_4$ such that $\Q(\alpha_i)=\Q(\beta_{\sigma(i)})$. This means that there exists a bijection between irreducible factors of $f$ and $g$ with the property that corresponding factors are related via a Tschirnhaus transformation. This immediately implies that $s_f$ and $s_g$ are conjugate.
\end{proof}

Lemma~\ref{lemma_conjugates} allows us to consider the Galois embedding problem independently of the choice of  polynomial, provided that the polynomials have the same splitting field and discriminant equal to $-3$ up to squares. Therefore, we may simply denote $\varphi_f$ by~$\varphi$.

We now prove the third statement of Theorem~\ref{main_theorem}; this will also establish the implication $1) \implies 2)$ in the second statement.

\begin{proof}[Proof of Theorem \ref{main_theorem} c).]
    Suppose that $E$ is an elliptic curve whose 3-division polynomial $\psi_{3,E}$ has the same splitting field as $f$. By Lemma \ref{lemma_conjugates}, $s_f$ and $s_{\psi_{3,E}}$ are conjugate since $\Delta_{\psi_{3,E}} \equiv -3 \pmod{\Q^{*2}}$. Let $\{x_i\}_{i=1}^4$ be the roots of $\psi_{3,E}$  and consider the non-trivial 3-torsion points of $E$:
$$ \pm P_1=(x_1,\pm y_1), \ \pm P_2=(x_2,\pm y_2), \ \pm (P_1+P_2)=(x_3,\pm y_3) \ \text{ and } \ \pm (P_1-P_2)=(x_4,\pm y_4).$$
As mentioned in Section \ref{sec_General_linear_group}, $\Phi$ was chosen in a way such that $\varrho_{3,E} \colon \Gal(\Qbar/\Q) \to \GL_2(\F_3)$ with respect to the basis $\{P_1,P_2\}$ of $E[3]$ satisfies $\pi \circ \varrho_{3,E} = \varphi \circ \pi.$
Therefore, $\varrho_{3,E} \colon \Gal(\Qbar/\Q) \to \widetilde{G}_i$ is a weak solution to the embedding problem. By Proposition \ref{prop_Zywina}, an appropriate quadratic twist of $E$ gives a proper solution.
\end{proof}

The second statement of Theorem \ref{main_theorem} is proved case by case. The surjective case $$\varphi(\Gal(K/\Q))=G_4 \cong S_4$$ was treated in \cite{LR}. In the remaining cases, a similar approach applies: we construct a non-CM elliptic curve $E$, given by a short Weierstrass model, whose $3$-division polynomial $\psi_{3,E}$ is a Tschirnhaus transformation of $f$.

\begin{proof}[Proof of the case $G_0 \cong C_2$]
First, notice that the obstruction associated to the Galois embedding problem is trivial since the sequence $$1 \longrightarrow \{\pm 1\} \longrightarrow \widetilde{G}_0 \longrightarrow G_0 \longrightarrow 1$$ is split.
Therefore, we have to prove that $K=\Q(\sqrt{-3})$ is the splitting field of 3-division polynomial of a non-CM elliptic curve $E$. The elliptic curve $E \colon y^2 = x^3 -6x-\frac{13}{4}$ is an example. In fact, as described in Proposition \ref{prop_Zywina}, $E_{t,C_2}$ is an infinite family of examples (only considering those without CM).
\end{proof}
\begin{proof}[Proof of the case $G_1 \cong C_2^2$.]
Given that $\Delta_f$ is not a square, we have that $f=g_1g_2$ where $g_1$ and $g_2$ are 2-degree irreducible polynomials. Hence, $K=\Q(\sqrt{\Delta_{g_1}},\sqrt{\Delta_{g_2}})=\Q(\sqrt{\Delta_{g_1}},\sqrt{-3})$ and 
$$ \Delta_{g_1}\Delta_{g_2} \equiv \Delta_f \equiv -3 \pmod{\Q^{*2}}.$$
Consider the automorphism $\sigma \in \Gal(K/\Q)$ sending $\sqrt{\Delta_{g_1}} \mapsto -\sqrt{\Delta_{g_1}}$, $\sqrt{-3} \mapsto \sqrt{-3}$ and the automorphism $\tau \in \Gal(K/\Q)$ sending $\sqrt{\Delta_{g_1}} \mapsto \sqrt{\Delta_{g_1}}$, $\sqrt{-3} \mapsto -\sqrt{-3}$. It is easy to check that $\varphi(\sigma)=\pm \begin{pmatrix}
    0 & 1 \\
    2 & 0
\end{pmatrix}$ lifts to an element of order $4$ in $\GL_2(\F_3)$. As discussed in \cite[p.\,283]{ledet2001embedding}, the obstruction of the embedding problem under consideration is 
$(\Delta_{g_1},3)$. 
Now, notice that the following assertions are equivalent: \begin{itemize}
    \item[1.] The Galois embedding problem is solvable.
    \item[2.] $(\Delta_{g_1},3)=1$.
    \item[3.] $\Delta_{g_1} \equiv t^2-6t-3= (t-3)^2-3\cdot 2^2 \pmod{\Q^{*2}}$ for some $t \in \Q^*$.
    \item[4.] $\Delta_{g_1} \equiv t^2-6t-3 \pmod{\Q^{*2}}$ for infinitely many $t \in \Q^*$.
    \item[5.] $K=\Q(\sqrt{\Delta_{g_1}}\sqrt{-3})=\Q(\sqrt{t^2-6t-3},\sqrt{-3})$ for infinitely many $t \in \Q^*$. 
\end{itemize}
As described in Proposition \ref{prop_Zywina}, the splitting field of the 3-division polynomial of the elliptic curve
$$E_{t,C_2^2}\colon y^2 = x^3-3(t - 3) (t + 1) (t^2 - 6t - 3)x-2(t^2 + 3) (t^2 - 6t - 3)^2$$
is a biquadratic field, namely $\Q(\sqrt{t^2-6t-3},\sqrt{-3})$. Therefore, there are infinite many $t \in \Q^*$ such that the splitting field of the 3-division polynomial of the elliptic curve $E_{t,C_2^2}$ is $K$. This implies that we can choose one without CM.
\end{proof}
\begin{proof}[Proof of the case $G_2 \cong S_3$] First, notice that the obstruction associated to the Galois embedding problem is trivial since the sequence 
$$1 \longrightarrow \{\pm1\} \longrightarrow \widetilde{G}_2 \longrightarrow G_2 \longrightarrow 1$$ is split. We have to prove that $K$ is the splitting field of the 3-division polynomial of a non-CM elliptic curve $E$. 

As described in Proposition \ref{prop_Zywina}, the splitting field of the 3-division polynomial of the elliptic curve
$$ E_{t,S_3} \colon y^2 = x^3 -3(t+1)(t+9)x-2(t+1)(t^2-18t-27)$$
is 
$$ \psi_{3,E_{t,S_3}}(x) = 3(x-3t-3)(x^3 + 3(t + 1)x^2 + 3(t+1)(t-15)x + (t+1)(t+9)^2).$$
By means of an appropriate linear transformation on the 3-degree factor of $\psi_{3,E_{t,S_3}}$, it is easy to check that $\psi_{3,E_{t,S_3}}$ has the same splitting field as 
$$\psi_{*}(x)= x^3+A(t)x+B(t),$$
where $A(t)=-3(t+1)$ and $B(t) = (t+1)(t+2)$.

Given that $f$ has a rational root, we can consider $a,b \in \Q$ such that $f_3(x)=x^3+ax+b$ has the same as the splitting field as $f$ and $ \Delta_{f_3} \equiv \Delta_f \equiv -3\pmod{\Q^{*2}}$. We want to prove that there exists $t \in \Q^*$ such that $\psi_{*}=f_{*}$ for some Tschirnhaus transformation $f_{*}$ of $f_3$.

Let $x_1,x_2,x_3$ be the roots of $f_3$. Fix $r \in \Q$ and define the Tschirnhaus transformation
$$ f_{*}(x)=(x-y_1)(x-y_2)(x-y_3)$$
where $y_i=\frac23an + rnx_i+nx_i^2$ for some $n \in \Q$. Then $$f_{*}(x)=x^3+a_*(r)n^2x+b_*(r)n^3$$ 
where $a_*(r) = \frac13(3ar^2 - a^2 + 9br)$ and $b_*(r)=\frac{1}{27} (-18a^2r^2 + 27br^3 - 2a^3 - 27abr - 27b^2)$.
Notice that the splitting field of $f_{*}$ is the same as the splitting field of $f_3$ whenever $n \neq 0$ since $\Q \neq \Q(y_i) \subset \Q(x_i)$ implies $\Q(y_i)=\Q(x_i)$.

There exists $t_0 \in \Q^*$ such that $A(t_0)=a_*(r)n^2$ and $B(t_0)=b_*(r)n^3$ if and only if $(a_*(r)n^2,b_*(r)n^3)$ is a zero of 
$$ P(x,y) = \text{Res}_t(A(t)-x,B(t)-y) = x^2-3x+9y.$$
Given that $a_*(r) \neq 0$ for all but finitely many every $r \in \Q$ and $$81b_*(r)^2+12a_*(r)^3 \equiv 81b^2+12a^3\equiv -3\Delta_{f_3} \equiv 1 \pmod{\Q^{*2}},$$ we can check that $P(a_*(r)n(r)^2,b_*(r)n(r)^3)=0$ where $n(r) = \frac{-9b_*(r) \pm \sqrt{81b_*(r)^2+12a_*(r)^3}}{2a_*(r)^2} \in \Q$. Defining $t_0(r)=-\frac{a_*(r)n(r)^2}{3}-1$ we have that $A(t_0(r))=a_*(r)n(r)^2$. This means that $E_{t_0(r),S_3}$ is an infinite family of elliptic curves whose 3-division polynomial is a Tschirnhaus transformation of $f_3$. From this family, one can select a curve without~CM. 
\end{proof}

\begin{proof}[Proof of the case $G_3 \cong D_4$.]
For this last proof we need to consider the following lemmas:

\begin{lemma}\label{lemma_same_splitting_field}
Let $f(x)$ be a four degree polynomial with $\Gal(f) \cong D_4$ and let $\alpha_1, \alpha_2, \alpha_3$ and $\alpha_4$ be the roots of $f$. If there are $a_0,a_1,a_2,a_3 \in \Q$ such that $$f_*(x):=\prod_{i=1}^4 (x-a_0-a_1\alpha_i-a_2\alpha_i^2-a_3\alpha_i^3)=x^4+ax^2+b$$
for some $a,b \in \Q$ satisfying that $b$ and $a^2-4b$ are square-free, then the splitting field of $f$ is isomorphic to the splitting field of $f_*$.
\end{lemma}
\begin{proof}
Notice that $f$ is irreducible since otherwise $|\Gal(f)| \leq 6$. Let $$f_*(x)=(x-\alpha)(x+\alpha)(x-\beta)(x+\beta).$$ If $(x-\alpha)(x+\alpha)=x^2-\alpha^2 \in \Q[x]$, then $\alpha^2$ is a rational root of $x^2+ax+b$. This contradicts the assumption that $a^2-4b$ is square-free. If $(x-\alpha)(x-\beta)=x^2-(\alpha+\beta)x+\alpha\beta \in \Q[x]$ then 
$$b \equiv \Delta_{f_*} =4(\alpha+\beta)^4(\alpha \beta)^2(\alpha-\beta)^4 \equiv 1 \pmod{\Q^{*2}}.$$
This contradicts the fact that $b$ is square-free. This means that $f_*$ is also irreducible and  $$\Q(\alpha,\beta)=\Q(\alpha_1,\alpha_2,\alpha_3,\alpha_4).$$
\end{proof}

\begin{lemma}\label{lemma:quadratic_forms}
Let $Q$ and $Q'$ be binary quadratic forms over a field $k$, and assume that $-\mathrm{disc}(Q')$ is not a square in $k$. Then the quartic curve
\[
C \colon Q'(x,y)^2 + Q(x,y)=0
\]
has a non-trivial $k$-rational point if and only if $Q$ represents $-1$ over $k$.\footnote{In fact, the map
$
(x,y)\longmapsto \left(\frac{x}{Q'(x,y)},\,\frac{y}{Q'(x,y)}\right)
$
defines an involution between the non-trivial $k$-rational points of $Q'(x,y)^2+Q(x,y)=0$ and the $k$-rational points of the conic $Q(x,y)=-1$.}
\end{lemma}

\comentario{
\begin{lemma}\label{lemma:quadratic_forms}
Let $Q$ and $Q'$ be two dimensional quadratic forms over a field $k$ such that $-\mathrm{disc}\, Q'$ is not a square in $k$. Then $C \colon Q'^2(x,y)+Q(x,y)=0$ has a non trivial rational point if and only if $Q$ represents -1. \footnote{In fact, the function $(x,y)\mapsto (\frac{x}{Q'(x,y)},\frac{y}{Q'(x,y)})$ is an involution between non trivial rational points of $Q'(x,y)^2+Q(x,y)=0$ and rational points of $Q(x,y)=-1$.}
\end{lemma}}

\begin{proof}
If $-\mathrm{disc}\, Q'$ is not a square then $Q'$ is anisotropic. If $Q'^2(x_0,y_0)+Q(x_0,y_0)=0$ then $$Q\left(\frac{x_0}{Q'(x_0,y_0)},\frac{y_0}{Q'(x_0,y_0)}\right)=-1.$$ Conversely, if $Q(x_0,y_0)=-1$ then $$Q'\left(\frac{x_0}{Q'(x_0,y_0)},\frac{y_0}{Q'(x_0,y_0)}\right)^2+Q\left(\frac{x_0}{Q'(x_0,y_0)},\frac{y_0}{Q'(x_0,y_0)}\right)=0.$$
\end{proof}

We may assume, by applying an appropriate Tschirnhaus transformation\footnote{This transformation is given by
$
f_*(x) = (x-(\alpha_1-\alpha_3))(x-(\alpha_2-\alpha_4))(x-(\alpha_3-\alpha_1))(x-(\alpha_4-\alpha_2)),
$
where $\alpha_1,\alpha_2,\alpha_3,\alpha_4$ are the roots of a polynomial $f$ with $\Delta_f \equiv -3 \pmod{\Q^{*2}}$.},
that $f(x)=x^4+ax^2+b$ for some $a,b \in \Q$, with $b \not\equiv 1 \pmod{\Q^{*2}}$ and $a^2-4b \not\equiv 1 \pmod{\Q^{*2}}$.
Indeed, since $\Delta_f \equiv -3 \pmod{\Q^{*2}}$, we may write $b=-3d^2$ for some $d \in \Q$, so that
$f(x)=x^4+ax^2-3d^2$.


We have that $K=\Q(\sqrt{\frac{-a+\sqrt{a^2+12d^2}}{2}},\sqrt{-3})$. As discussed in \cite[p.~20]{grundman1995groups}, the obstruction to the embedding problem is $(3(a^2+12d^2),2a)$ if $a \neq 0$ and trivial if $a=0$.

Suppose that the embedding problem is solvable. We want to show that there exists an elliptic curve $E$ without CM such that the splitting field of the 3-division polynomial of $E$ is $K$. As described in Proposition \ref{prop_Zywina}, this elliptic curve may take the form (up to a quadratic twist)
$$ E_{t,D_4} \colon y^2 = x^3 -3(t+1)(t+9)x-2(t+1)(t^2-18t-27),$$
for some $t \in \Q$.

Let $x_1,x_2,x_3$ and $x_4$ be the roots of $\psi_{3,E_{t,D_4}}$. By Lemma \ref{lemma_same_splitting_field}, we can obtain a new polynomial with the same splitting field as $\psi_{3,E_{t,D_4}}$ over $\Q(t)$ via the Tschirnhaus transformation
$$ \psi_{*}(x) = \prod_{i=1}^4(x-y_i)= 9(x^4+A(t)x^2-3D(t)^2)$$
where $y_i = -\frac{t(t-12)}{8}-\frac{x_i}{12}+\frac{x_i^2}{24(t^2+12t+144)}$, $A(t) = -6(t - 12) (t + 24)$ and $D(t)=3t(t-12)$ since $$A(t)^2+12D(t)^2 \equiv t^2+12t+144 \not\equiv 1 \pmod{\Q(t)^{*2}}.$$ Let us prove that there exists $t_0 \in \Q^*$ such that the specialization of $\psi_*$ at $t=t_0$ is equal to $f_*$ for some Tschirnhaus transformation $f_*$ of $f$. 

Let $\alpha_1,\alpha_2,\alpha_3$ and $\alpha_4$ be the roots of $f$. Define the Tschirnhaus transformation of $f$ as
$$f_*(x) = \prod_{i=1}^4 (x-\beta_i)$$
where $\beta_i = n\alpha_i^3+m\alpha_i$ with $(n,m) \in \Q^2-\{0\}$. A straightforward computation shows that 
$$ f_*(x) = x^4+a_*(n,m)x^2-3d_*(n,m)^2$$ where $$a_*(n,m) = (a^3 + 9ad^2)n^2 - (2a^2 + 12d^2)mn + am^2,$$ $$d_*(n,m) = 3d^3n^2 + admn - dm^2.$$ Moreover, by Lemma \ref{lemma_same_splitting_field}, the polynomial $f_*$ has the same splitting field as $f$ since\footnote{Here we are using the fact that $a^2n^2 + 3d^2n^2 - 2amn + m^2$ has no non-trivial rational solution.} $$a_*(n,m)^2+12d_*(n,m)^2 = (a^2n^2 + 3d^2n^2 - 2amn + m^2)^2(a^2 + 12d^2) \equiv a^2+12d^2 \not\equiv 1 \pmod{\Q^{*2}}.$$

Given $(n,m) \in \Q^2$, there exists $t_0 \in \Q$ such that $A(t_0)=a_*(n,m)$ and $D(t_0)=d_*(n,m)$ if and only if $(a_*(n,m),d_*(n,m))$ is a zero of $$P(x,y) = \frac{1}{9}\mathrm{Res}_t(x-A(t),y-D(t)) = (x+2y)^2 - 1728(x + 6y).$$
Hence, it suffices to show that the curve $$C_{a,d} \colon (a_*(n,m)+2d_*(n,m))^2-1728(a_*(n,m)+6d_*(n,m))=0$$
has a non-trivial rational point.

We can apply Lemma \ref{lemma:quadratic_forms} to
$$Q_{a,d}(n,m) = -1728(a_*(n,m)+6d_*(n,m))$$
and
$$Q'_{a,d}(n,m) = a_*(n,m)+2d_*(n,m),$$
since $-\mathrm{disc}\, (a_*+2d_*) = 4d^2(a^2+12d^2) \equiv  a^2+12d^2 \not\equiv 1 \pmod{\Q^{*2}}$. As a result we get that $C_{a,d}$ has a non-trivial rational point if and only if $-1728(a_*+6d_*)$ represents $-1$ or, equivalently, $1728(a_*+6d_*)$ represents 1. 

\noindent If $a^3 + 9ad^2 + 18d^3 \neq 0$, then $1728(a_*+6d_*)$ represents $1$ if and only if\footnote{Recall that a quadratic form $Q(x,y)=Ax^2+By^2+Cxy$ with $B\neq 0$ represents 1 if and only if $(-\mathrm{disc} \, Q,B)=1$.} $$(-\mathrm{disc}\, (1728(a_*+6d_*))), 1728(a^3 + 9ad^2 + 18d^3)) = 1.$$ But
\begin{align*}
    (-\mathrm{disc}\, (1728(a_*+6d_*)))&, 1728(a^3 + 9ad^2 + 18d^3))= \\ 
    &= (3(a^2+12d^2), 3(a^3 + 9ad^2 + 18d^3)) \\ &= (3(a^2+12d^2), 2a) \\ &=1,
\end{align*}
where the second equality follows from $(3(a^2+12d^2), 2a)(3(a^2+12d^2), 3(a^3 + 9ad^2 + 18d^3))=(3(a^2+12d^2), 6a(a^3 + 9ad^2 + 18d^3))=1$ since
$$3(a^2+12d^2)(a-3d)^2 + 6a(a^3 + 9ad^2 + 18d^3) = (3(a^2 - ad + 6d^2))^2.$$
If $a^3 + 9ad^2 + 18d^3 = 0$, then $a-6d \neq 0$ and $1728(a_*+6d_*)$ represents $1$ if and only if $$(-\mathrm{disc}\, (1728(a_*+6d_*)), 1728(a-6d)) = 1.$$ But
\begin{align*}
    (-\mathrm{disc}\, (1728(a_*+6d_*))&, 1728(a-6d))= \\ 
    &= (3(a^2+12d^2), 3(a-6d)) \\ &= (3(a^2+12d^2), 2a) \\ &=1,
\end{align*}
where the second equality follows from $((3(a^2+12d^2), 6a(a-6d)))=1$ since $$3(a^2+12d^2) + 6a(a-6d) = (3a-6d)^2.$$
If $a=0$, then the existence of a non-trivial rational point on $C_{a,d}$ follows from $$-\mathrm{disc}\, 1728(a_*+6d_*) \equiv 36d^2 \equiv 1 \pmod{\Q^{*2}}.$$
Finally, given that the set of rational points $C_{a,d}(\Q)$ is in a bijection with the set of rational points of a conic, we conclude that $C_{a,d}(\Q)$ is infinite. For each $(n,m) \in C_{a,d}(\Q)$ we can find $t \in \Q$ such that $A(t)=-6(t-12)(t+24)=a_*(n,m)$ and $D(t)=3t(t-12)=d_*(n,m)$. We get $$t=\frac{1728-a_*(n,m)-2d_*(n,m)}{144}.$$ 

Suppose that $(a_*+2d_*)(C_{a,d}(\Q))$ is a finite set. Then $(a_*+6d_*)(C_{a,d}(\Q))$ is also finite. This implies that $a_*(C_{a,d}(\Q))$ and $d_*(C_{a,d}(\Q))$ are both finite. But
$$ 3(a^2+12d^2)^2d^4n^4+(a^2+12d^2)(aa_*-12dd_*)d^2n^2+(ad_*+a_*d)^2=0$$
and $$ 3(a^2+12d^2)^2d^4m^4+(a^3a_* + 9aa_*d^2 + 36d_*d^3)(a^2 + 12d^2)d^2m^2-(a^3d_* + 9ad_*d^2 - 3a_*d^3)^2=0$$
hold for every $(n,m) \in\Q^2$. Hence, for each $C_1,C_2 \in \Q$, the conics $a_*(n,m)=C_1$ and $d_*(n,m)=C_2$ intersect\footnote{This is essentially Bézout's Theorem.} in finitely many $(n,m) \in \Q^2$. This contradicts that $C_{a,d}(\Q)$ is infinite.
Thus, $(a_*+2d_*)(C_{a,d}(\Q))$ is infinite and we can choose $t$ such that $E_t/\Q$ ($j(E_t)=t^3$) does not have CM.
\end{proof}

\begin{corollary}
The proof of Theorem~\ref{main_theorem} also shows that, whenever any of the considered embedding problems is solvable, there exist infinitely many elliptic curves (with pairwise distinct $j$-invariants) whose mod~$3$ Galois representations provide weak solutions. Moreover, the proof yields an explicit method for constructing such elliptic curves.
\end{corollary}

\comentario{
\begin{corollary}
The proof of Theorem \ref{main_theorem} also shows that, in case that any of the embedding problems considered is solvable, there exists infinitely many elliptic curves (with distinct $j$-invariant) whose mod 3 Galois representation provides a weak solution. In fact, we have described an explicit method to find these elliptic curves.
\end{corollary}
}

\section{Examples}\label{sec_examples}
We conclude this article with several examples illustrating Theorem \ref{main_theorem}. In each solvable example, we find an infinite family of elliptic curves solving the problem by the method described. 

\textit{Case $G_0 \cong C_2$.} Let $f(x)=x^2(x^2+3)$ with $\Delta_f \equiv -3 \pmod{\Q^{*2}}$. The splitting field of $f$ is $K=\Q(\sqrt{-3})$ and its Galois group is isomorphic to $G_0 \cong C_2$. Each non-CM elliptic curve from the family $\{E_{t,C_2}, t \in \Q^*\}$ described in Proposition \ref{prop_Zywina} has 3-division polynomial with splitting field $\Q(\sqrt{-3})$. 

\textit{Case $G_1 \cong C_2^2$ (obstructed).} Let $f(x)=(x^2-2)(x^2+6)$ with $\Delta_f \equiv -3 \pmod{\Q^{*2}}$. The splitting field of $f$ is $\Q(\sqrt{2},\sqrt{-6})$, and its Galois group is isomorphic to $G_1 \cong C_2^2$. Since $(2,3)=-1$, the Galois embedding problem is obstructed. Thus, by Theorem \ref{main_theorem}, there exists no non-CM elliptic curve $E/\Q$ whose 3-division polynomial has splitting field $\Q(\sqrt{2},\sqrt{-6})$.

\textit{Case $G_1 \cong C_2^2$ (unobstructed).} Let $f(x)=(x^2+2)(x^2-6)$ with $\Delta_f \equiv -3 \pmod{\Q^{*2}}$. The splitting field of $f$ is $K=\Q(\sqrt{-2},\sqrt{6})$, and its Galois group is isomorphic to $G_1 \cong C_2^2$. Since $(-2,3)=1$, the Galois embedding problem is unobstructed. Thus, by Theorem \ref{main_theorem} there exists a non-CM elliptic curve $E/\Q$ such that the splitting field of the 3-division polynomial of $E$ is $K$. Indeed, we can explicitly find infinitely many such elliptic curves by computing $t \in \Q$ satisfying
$$-2 \equiv t^2-6t-3 \pmod{\Q^{*2}}.$$
A straightforward computation shows that 
$$t(r) = \frac{r^2-8r+10}{r^2+2}$$
with $r \in \Q$ is a parametrization. This means that the 3-division polynomial of every non-CM elliptic curve $E$ with $j(E)=27\frac{(t(r)-3)^3(t(r)+1)^3}{t(r)^3}$ has splitting field $\Q(\sqrt{-2},\sqrt{6})$.

\textit{Case $G_2 \cong S_3$.} Let $f(x)=x(x^3+2)$ with $\Delta_f \equiv -3 \pmod{\Q^{*2}}$. The splitting field of $f$ is $K=\Q(\sqrt[3]{2},\sqrt{-3})$, and its Galois group is isomorphic to $G_2 \cong S_3$. The obstruction to the embedding problem is trivial. Thus, by Theorem \ref{main_theorem} there exists a non-CM elliptic curve $E/\Q$ such that the splitting field of the 3-division polynomial of $E$ is $K$. Indeed, we can explicitly find infinitely many such elliptic curves by computing 
$$ t(r)=-\frac{-a_*(r)n(r)^2}{3}-1=-\frac{2+r^3}{r^3}.$$
This means that the 3-division polynomial of every non-CM elliptic curve $E$ with $j(E)=27\frac{(t(r)+1)(t(r)+9)^3}{t(r)^3}$ has splitting field $\Q(\sqrt[3]{2},\sqrt{-3})$.

\textit{Case $G_3 \cong D_4$ (obstructed).} Let $f(x)=x^4+x^2-3=x^4+ax^2-3d^2$ with $\Delta_f \equiv -3 \pmod{\Q^{*2}}$. The splitting field of $f$ is $K=\Q(\sqrt{-\frac{1}{2}+\frac{1}{2}\sqrt{13}},\sqrt{-3})$, and its Galois group is isomorphic to $G_3 \cong D_4$. The obstruction to the embedding problem is $$(3(a^2+12d^2),2a)=(39,2)=-1.$$
Thus, by Theorem \ref{main_theorem}, there exists no non-CM elliptic curve $E/\Q$ whose 3-division polynomial has splitting field $\Q(\sqrt{-\frac{1}{2}+\frac{1}{2}\sqrt{13}},\sqrt{-3})$.

\textit{Case $G_3 \cong D_4$ (unobstructed).} Finally, let $f(x) = x^4+2x^2-12=x^4+ax^2-3d^2$ with $\Delta_f \equiv -3 \pmod{\Q^{*2}}$.
This polynomial has splitting field $K=\Q(\sqrt{-1+\sqrt{13}},\sqrt{-3})$ and Galois group isomorphic to $G_3 \cong D_4$. The obstruction to the embedding problem is $$ (3 \cdot 52, 4) =(156,1)=1.$$
Thus, by Theorem \ref{main_theorem} there exists a non-CM elliptic curve $E/\Q$ such that the splitting field of the 3-division polynomial of $E$ is $K$. Indeed, we can explicitly find infinitely many such elliptic curves by computing the rational points of 
$$C_{2,2} \colon (a_*(n,m)+2d_*(n,m))^2-1728(a_*(n,m)+6d_*(n,m))=0,$$
where 
$$a_*(n,m) = 80n^2-56nm+2m^2,$$
$$d_*(n,m) = 24n^2+4nm-2m^2.$$
Following the proof of the theorem, this can be done by first computing the rational points of the conic $1728(a_*(n,m)+6d_*(n,m))=1$. As a result we get that
$$n(r) = \frac{-6(5r^2 + 40r + 176)(5r^2 + 16r - 112)}{23r^4 - 144r^3 - 2592r^2 - 14080r - 2304},$$
$$m(r) = \frac{-24(5r^2 + 16r - 112)(r^2 + 56r + 112)}{23r^4 - 144r^3 - 2592r^2 - 14080r - 2304}$$ parametrize the rational points of $C_{2,2}$. For every $r \in \Q$ we can find $t(r) \in \Q$ such that $A(t(r))=-6(t(r)-12)(t(r)+24)=a_*(n(r),m(r))$ and $D(t(r))=3t(r)(t(r)-12)=d_*(n(r),m(r))$. Solving these quadratic equations yields
$$t(r) = \frac{1728-a_*(n(r),m(r))-2d_*(n(r),m(r))}{144} = \frac{-12(27r^4 + 464r^3 + 864r^2 + 6912r + 27392)}{23r^4 - 144r^3 - 2592r^2 - 14080r - 2304}.$$
This means that the 3-division polynomial of every non-CM elliptic curve $E$ with $j(E)=t(r)^3$ has splitting field $\Q(\sqrt{-1+\sqrt{13}},\sqrt{-3})$.

\bibliography{main}{}

@article{Zywina,
author = {Zywina, David},
year = {2015},
month = {08},
pages = {},
title = {On the possible images of the mod $\ell$ representations associated to elliptic curves over~$\bold{Q}$},
doi = {10.48550/arXiv.1508.07660},
journal={DOI: 10.48550/arXiv.1508.07660},
}

@article {GL,
    AUTHOR = {Gonz\'alez, Josep and Lario, Joan-C.},
     TITLE = {Rational and elliptic parametrizations of {$\bold Q$}-curves},
   JOURNAL = {J. Number Theory},
  FJOURNAL = {Journal of Number Theory},
    VOLUME = {72},
      YEAR = {1998},
    NUMBER = {1},
     PAGES = {13--31},
     }

@article{LR,
  title={An octahedral-elliptic type equality in {$Br_2(k)$}},
  author={Lario, J-C and Rio, Anna},
  journal={Comptes rendus de l'Acad{\'e}mie des sciences. S{\'e}rie 1, Math{\'e}matique},
  volume={321},
  number={1},
  pages={39--44},
  year={1995}
}

@article{ledet2001embedding,
  title={Embedding problems and equivalence of quadratic forms},
  author={Ledet, Arne},
  journal={Mathematica Scandinavica},
  pages={279--302},
  year={2001},
  publisher={JSTOR}
}

@article{grundman1995groups,
  title={Groups of order 16 as Galois groups},
  author={Grundman, Helen G and Smith, Tara L and Swallow, John R},
  journal={Expo. Math},
  volume={13},
  pages={289--319},
  year={1995}
}

@article{bandini2016fields,
  title={Fields generated by torsion points of elliptic curves},
  author={Bandini, Andrea and Paladino, Laura},
  journal={Journal of Number Theory},
  volume={169},
  pages={103--133},
  year={2016},
  publisher={Elsevier}
}
\bibliographystyle{alpha}

\end{document}